\documentclass[12pt]{article}
\usepackage[utf8]{inputenc}

\usepackage{amsmath,amssymb,amsthm,bbm}
\usepackage{mathtools}
\usepackage{color}

\usepackage{tikz}
\usetikzlibrary{matrix}
\usetikzlibrary{angles,quotes}

\theoremstyle{definition} 
\theoremstyle{definition} 
\theoremstyle{definition} \newtheorem*{axi*}{Axiom}
\theoremstyle{definition} \newtheorem{prp}{Proposition}[section]
\theoremstyle{definition} 
\theoremstyle{definition} \newtheorem{lem}[prp]{Lemma}
\theoremstyle{definition} \newtheorem{thm}[prp]{Theorem}
\theoremstyle{definition} 
\theoremstyle{definition} 
\theoremstyle{remark} 

\theoremstyle{remark}

\newcommand \Q {\mathbb Q}

\newcommand \C {\mathbb C}
\newcommand \F {\mathbb F}

\makeatletter
\newcommand{\subjclass}[2][1991]{%
  \let\@oldtitle\@title%
  \gdef\@title{\@oldtitle\footnotetext{#1 \emph{MSC classes.} #2}}%
}
\newcommand{\keywords}[1]{%
  \let\@@oldtitle\@title%
  \gdef\@title{\@@oldtitle\footnotetext{\emph{Keywords:} #1.}}%
}
\makeatother

\title{Special directions on the finite affine plane}
\author{Gergely Kiss\thanks{Alfr\'ed R\'enyi Institute of Mathematics, e-mail: kigergo57@gmail.com}, G\'abor Somlai \thanks{On an unpaid leave at E\"otv\"os Lor\'and University, Department of Algebra and Number Theory and Alfr\'ed Rényi Institute of Mathematics, A Fulbright scholar at the Graduate Center of the City University of New York, e-mail: gabor.somlai@ttk.elte.hu}}
\date{}

\begin{document}



\maketitle
\begin{abstract}
    In this paper we study the number of special directions of sets of cardinality divisible by $p$ on a finite plane of characteristic $p$, where $p$ is a prime. We show that there is no such a set with exactly two special directions. We characterise sets with exactly 3 special directions which answers a question of Ghidelli in negative. Further we introduce methods to construct sets of minimal cardinality that has exactly 4 special directions for small values of $p$.
\end{abstract}
\section{Introduction}
Let $S$ be a set of points in the affine plane over the finite field of $p$ elements, where $p$ is a prime.
A natural way of viewing points of the projective line is to consider the equivalence classes of the nonzero vectors of $\mathbb{F}_p^2$. We write $u \sim v$ if $u$ is a nonzero multiple of $v$ and the equivalence class of $u$ which we call the \textit{direction} of $u$ is denoted by $d(u)$. We say that a pair of vectors $w_1 \ne w_2 \in \mathbb{F}_p^2$ determines the direction $d(u)$ if $u \sim w_1-w_2$. Finally, we denote by $D(S)$, the set of directions determined by the pair of points in $S$.

 By elementary pigeonhole argument one can see that every direction is determined by any subset of $\mathbb{F}_p^2$ of cardinality at least $p+1$. It was proved by R\'edei that if $S$ is of cardinality $p$, then either $S$ is a line or $S$ determines at least $\frac{p+3}{2}$ directions, see \cite{redei}.
 The same result was independently proved by Dress, Klin and Muzychuk \cite{DKM}. As a corollary of their argument they obtained a new proof for Burnside's classical theorem on permutation groups of prime degree.

 This was generalised by Sz\H{o}nyi, who proved that if $|S|<p$ and $S$ is not contained in a line, then $|D(S)|\ge \frac{|S|+3}{2}$. Lov\'asz and Schrijver \cite{LS} showed that if $|D(S)|=\frac{|S|+3}{2}$ for some $|S|=p$, then $S$ is an affine transform of the graph of the function $f(x)=x^{\frac{p+1}{2}}$.
 In \cite{gacs} G\'acs proved that the $|D(S)|$ cannot be between $\frac{{p + 5}} {2}$ and $2\frac{{p - 1}} {3}$ and showed that the upper bound obtained is one less than the smallest known example.

Another way of thinking of directions of $p$-element subsets of $\mathbb{F}_p^2$ is the following. We say that $S$ is \textit{equidistributed} in a direction if $S$ intersects the lines having the corresponding fixed slope in the same amount of points.
Equidistributivity is one of the key tools in investigating spectral sets of finite abelian groups, see \cite{covenmeyerowitz},\cite{KMSV},\cite{laba}.
One can also see that for sets $|S|=p$ we have that $S$ is equidistributed in the direction $m$ if and only if $m \not\in D(S)$.
In general, equidistributivity of a set of a certain direction implies that $p \mid |S|$. This motivates that we only study sets whose cardinality is a multiple of $p$ in the remaining part of the paper.

The investigation of sets of cardinality larger than $p$ was initiated by Ghidelli \cite{ghidelli}. It was proved that a set of cardinality $kp$ ($1 \le k \le p,~ k \in \mathbb{Z}$) is either a set of parallel lines or is not equidistributed in at least $\lceil \frac{p+k+2}{k+1}\rceil$ directions. From now on we call a direction \textit{special} if $S$ is not equidistributed in that direction.
Note that Ghidelli's definition of special direction is more general than this one and his results also handle sets whose cardinality is not divisible by $p$. For sets of cardinality divisible by $p$ the two definitions of special directions coincide.
It was asked by Ghidelli, whether the sets, which are not the union of a set of parallel lines, determine at least $\frac{p+3}{2}$ special directions.

The main purpose of this paper is to construct an example to answer Ghidelli's problem in the negative. We prove the following theorem.
\begin{thm}
Up to an affine transformation, there is a unique set $S$ of size $\frac{p(p-1)}{2}$ in $\mathbb{F}_p^2$, which is equidistributed in $p-2$ directions. Moreover, every set having exactly 3 special directions can be transformed by an affine transformation (elements of $AGL(2,p)$) to either $S$ or $S^c$, where $S^c$ is the complement of $S$ in $\mathbb{F}_p^2$.
\end{thm}
This shows that for $k=\frac{p-1}{2}$ the result of Ghidelli is tight. A natural question arises here. Is it possible to construct sets of cardinality $kp$ which have $\lceil \frac{p+k+2}{k+1}\rceil$ special directions?

The paper is organised as follows.
In Section \ref{sec2} we describe sets having at most $2$ special directions.
Section \ref{sec3} is devoted to the analysis of the proof of Ghidelli in order to understand the possible ways of describing examples for his problems. Then in Section \ref{sec4} we describe sets having 3 special directions while in Section \ref{sec6} we present some examples for sets having 4 special directions. Section \ref{sec5} contains a reformulation of the problem. Finally, we raise some questions concerning the topic in Section \ref{sec7}.
\section{Two special directions}\label{sec2}
From now on, let $\mathbb{F}_p^2$ be identified with the set of pairs of integers $(a,b)$, where $a,b \in \{ 0,1, \ldots, p-1\}$.
Let us assume that $S$ is a subset of $\mathbb{F}_p^2$ of cardinality $kp$ which is equidistributed in $p-1$ directions. It has been proved by Fallon, Mayeli and Villano \cite{Tom} that in this case $S$ is the union of $k$ parallel lines.
This also means that a set having at most two special directions has at most one.
The original proof is short but uses techniques from Fourier analysis. We present a combinatorial argument for the statement.

Let us assume that $S$ is equidistributed along the lines $l(x)=ax+b$, where $a \in \mathbb{F}_p^*$, $b \in \mathbb{F}_p$.
We may assume that $(y,c) \in S$ and $(z,c) \not\in S$ for some $c \in \mathbb{F}_p$ and $y \ne z \in \mathbb{F}_p$.  If there is no such pair, then $S$ is the union of $k$ lines. We may assume $c=0$ and $z=0$ since $D(S)=D(S+t)$, where $t \in \mathbb{F}_p^2$.
Let $l_{\infty}^j= \{ (j,i) \mid i \in \mathbb{F}_p\}$ and $l_0=\{(x,0) \mid x \in \mathbb{F}_p\}$.

Now we count the cardinality of $S$ in three different ways.
First, $|S|=kp$. It is equal to the number of points contained on the lines containing $(0,0) \not\in S$. We obtain
\begin{equation}\label{eq1} kp=k(p-1)+(a_0+b_0), \end{equation}
where $a_0=|S \cap l_0|$ 
and $b_0=|S \cap l_{\infty}^0|$. 

On the other hand we may count the number of elements contained in the lines going through $(y,0) \in S$.
\begin{equation}\label{eq2} kp=1+ (k-1)(p-1)+(a_0-1)+(b_1-1), \end{equation}
where $b_1=|S \cap l_{\infty}^y|$.
It follows from equation \eqref{eq1} that $a_0+b_0=k$. In particular we get $a_0 \le k$. Equation \eqref{eq2} shows $k+p=a_0+b_1$. Plainly, $b_1 \le p$ and we have seen $a_0 \le k$ so $b_1 = p,a_0 = k$ and $b_0=0$. This shows that $l_{\infty}^y$ is contained in $S$. This holds for every $y \in \mathbb{F}_p$ with $(y,0) \in S$. Since $a_0=k$ we have $k$ such $y$. We have found the $kp$ elements of $S$ and hence $S$ is the union of parallel (vertical) lines.
\section{Ghidelli's proof}\label{sec3}
In this section we follow Ghidelli's proof to obtain some extra information about sets having few special directions.

Let $S \subseteq \mathbb{F}_p^2$ be a nonempty set of cardinality $np$. We define the R\'edei polynomial associated with $S$ as $$H_S(x,y)=\prod_{(a,b)\in S \subseteq \mathbb{F}_p^2} (x-ay+b).$$ One of the main properties of $H_S$ follows from the observation that \[ (x - ay + b) = (x - a'y + b') \Longleftrightarrow \frac{b-b'}{a - a'}
= y.\]
In other words, if $y=m$ is fixed, then
$am - b=c$ for a given $c\in \mathbb{F}_p$ holds for those points $(a,b)$ of $\mathbb{F}_p^2$ which are contained in a line whose slope is $m$. Therefore if we pick a set of representatives $\{(a_i,b_i) \mid i=0,1,\ldots, p-1\}$ of the class of parallel lines of slope $m$ (one point from each line), then $\prod_{i=0}^{p-1} x-a_im+b_i=x^p-x$.  It follows that if $S$ is equidistributed in the direction $m$, then $H_S(x,m)=(x^p-x)^n$.

We may write \[ H_S(x,y)=x^{|S|}+x^{np-1}g_1(y)+ x^{np-2}g_2(y)+ \ldots +g_{np},
\]
where $g_i \in \mathbb{F}_p[y]$ ($i=1, \ldots, np$).
It is easy to see that $g_i$ is of degree at most $i$ and it is equal to the $i$'th elementary symmetric polynomial\footnote{For each nonnegative integer $k$, the $k$'th elementary symmetric polynomial on $n$ variables is the sum of all distinct products of $k$ distinct variables. We denote it by $\sigma_k$.} $$g_i(y)=\sigma_i(a_1y-b_1,a_2y-b_2, \ldots , a_{np}y-b_{np})$$ of the polynomials  $a_jy-b_j$ ($j=1, \dots, np$).

Now assume that $S$ is equidistributed in $p-1 \ge$  directions. Then $g_i(y)$ has at least $k$ roots since the coefficients $x^{a}$, where $a>(n-1)p+1$ in the polynomial $(x^p-x)^n$ are $0$. Then we have $g_i\equiv 0$ if $i<k$ since the number of its roots is larger than its degree. By Newton's identities $\sum_{i=1}^{np}(a_iy+b_i)^l=0$ if $l<k$ and so as the leading coefficients of these polynomials $\sum_{i=0}^{np}a_i^l$, these expressions should also be $0$.

Let $w_j$ be the number of indices $i$ such that $a_i=j$. Then
\[ \sum_{i=0}^{np}a_i^l =\sum_{j=0}^{p-1}w_j j^l.
\]
This shows that the vector $w=(w_j)_{j=0,1,\ldots, p-1}$ is orthogonal to $(j^l)_{j=0,1,\ldots, p-1}$ in $\mathbb{F}_p^p$ for $l=1,\ldots, k$. Further $w \in \mathbb{F}_p^p$ is orthogonal to $(1)_{j=0,1,\ldots, p-1}$ since $|S|$ is divisible by $p$. Thus $w$ is orthogonal to the first $k$ rows of the Vandermonde matrix $M_{j,l}=(j^l)$ ($0 \le j,l \le p-1$ ).

It is not hard to see that the $i$'th and $j$'th rows of $M$ are orthogonal in $\mathbb{F}_p^p$ except if $i+j=p-1$. Thus we obtain that the orthogonal subspace for $\langle 1,x,x^2, \ldots, x^k \rangle$ is $\langle 1,x,x^2, \ldots, x^{p-1-1-k} \rangle$. As a corollary of this argument we obtain the following.
\begin{prp}
Let $w$ be the projection function associated with $S$ defined above. Assume $S$ has $k \ge 2$ special directions. Then $w$ as a function from $\F_p$  to $\F_p$ can be expressed as a polynomial of degree at most $k-2$.
\end{prp}
As a corollary of this argument we obtain again that sets having exactly 2 special directions do not exist since the projections are constant functions (after deleting lines) so every direction is equidistributed or the set is the union of parallel lines. Both cases contradict the fact that the set has 2 special directions.

Assume now that we have a set $S$, which is equidistributed in $p-2$ directions but it is not the union of lines. In this case $a_i$ is a function that is either constant or linear. We may write it as $\alpha i + \beta$.

The $\alpha=0$ case is realised when $S$ is the union of parallel lines.

In the case when $a_i=\alpha i +\beta$ with $\alpha \ne 0$, we obtain in particular that $|S| = 0+1+\ldots + p-1=\frac{p(p-1)}{2}$ or $|S| = 1+\ldots + p-1+p=\frac{p(p+1)}{2}$ since linear polynomials are permutation polynomials.

\section{Sets with 3 special directions}\label{sec4}
Using the result of the previous section we present a natural construction that fulfils the required conditions to answer Ghidelli's question in the negative. What is more, we prove that the sets (and their images of $AGL(2,p)$) described in this section are the ones having exactly 3 special directions.
We emphasise the fact that we think of the coordinates as elements of $\mathbb{Z}$, which gives us the opportunity to compare them. However, the additive and multiplicative operations are understood ($\bmod ~ p$).

According to the observations in the previous section it seems reasonable to try to understand the properties of the following set:
\[ S=\{(a,b) \in \mathbb{F}_p^2 \mid b<a \}. \]
$S$ has $\frac{p(p-1)}{2}$ elements. (See also Figure  \ref{fig:triangle}.)

\begin{figure}
\begin{center}
\begin{tikzpicture}
\matrix[matrix of nodes,nodes={draw=gray, anchor=center, minimum size=.5cm}, column sep=-\pgflinewidth, row sep=-\pgflinewidth] (A) at (-1,0) {
0 & 0 & 0 & 0 & 0 &0 & 0 & 0 & 0 & 0 &0 \\
0 & 0 & 0 & 0 & 0 &0 & 0 & 0 & 0 & 0 &1 \\
0 & 0 & 0 & 0 & 0 &0 & 0 & 0 & 0 & 1 &1 \\
0 & 0 & 0 & 0 & 0 &0 & 0 & 0 & 1 & 1 &1 \\
0 & 0 & 0 & 0 & 0 &0 & 0 & 1 & 1 & 1 &1 \\
0 & 0 & 0 & 0 & 0 &0 & 1 & 1 & 1 & 1 &1 \\
0 & 0 & 0 & 0 & 0 &1 & 1 & 1 & 1 & 1 &1 \\
0 & 0 & 0 & 0 & 1 &1 & 1 & 1 & 1 & 1 &1 \\
0 & 0 & 0 & 1 & 1 &1 & 1 & 1 & 1 & 1 &1 \\
0 & 0 & 1 & 1 & 1 &1 & 1 & 1 & 1 & 1 &1 \\
0 & 1 & 1 & 1 & 1 &1 & 1 & 1 & 1 & 1 &1 \\
};
\end{tikzpicture}
\end{center}
\caption{Set $S$ which has 3 special directions.}
\label{fig:triangle}
\end{figure}

Clearly, $S$ is not equidistributed in at least $3$ directions  since the lines having equation $ax+by=c$, where $(a,b,c) \in \mathbb{F}_p^3$ is either $(1,0,0)$, $(0,1,p-1)$ or $(1,-1,1)$ intersects $S$ in $p-1>\frac{p-1}{2}$ elements.



Let $L$ be a line containing the origin. We show that if the equation determining $L$ is $f_L(x)=ax$ with $ 2 \le a \le p-1$, then $|L \cap S|=\frac{p-1}{2}$. Clearly, $(0,0) \not\in S$ and if $i<ai$ for some $i \in \{1, \ldots, p-1 \}$, then $-i>a(-i)$ since $a \ne 0$. Note also that $i \ne ai$ since $i \ne 0$ and $a \ne 1$.

It remains to verify that if $|S\cap (L+i)|=\frac{p-1}{2}$, then $|S\cap (L+i+1)|=\frac{p-1}{2}$. Therefore, we show there is exactly one $j \in \{ 0, \ldots, p-1\}$ such that $f_L(j)+i >j$ and $f_L(j)+i+1 < j$. This happens when $f_L(j)+i=aj+i=p-1$ and $j \ne 0$. Since $a\ne 0$, if $i = p-1$ we would get $j =0$, which is excluded. If $i \ne p-1$, then there is a unique $j=\frac{p-1-i}{a}$ fulfilling the equation. Thus, there is a unique column, when the intersection of $S$ with the line $L+i$ increases by $1$, when we replace the line $L+i$ by $L+i+1$.

On the other hand, if $f_L(j)+i=j-1$ ($j \ne 0   $), then $(j,f_L(j)+i)$, which is an element of $L+i$ is in $S$ but  $(j,f_L(j)+i+1) \not\in S$. The solution of the equation $f_L(j)+i=aj+i=j-1$ is $j=-\frac{i+1}{a-1}$. Note that $a \ne 1$ so such a $j$ exists.

The case $j=0$ ($i=p-1$) can be handled similarly.
\begin{thm}
Let $T$ be a subset of $\mathbb{F}_p^2$ which is equidistributed in $p-2$ directions. Then $T=\alpha(S)$, if $|T|=\frac{p(p-1)}{2}$ and $T=\alpha'(S^c)$ if $|T|=\frac{p(p+1)}{2}$ for some $\alpha,\alpha' \in AGL(2,p)$, where $S^c$ is the complement of $S$ in $\mathbb{F}_p^2$.
\end{thm}
\begin{proof}
We have seen that $|T|=\frac{p(p-1)}{2} \mbox{ or } \frac{p(p+1)}{2}$. As the complement of a set of size $\frac{p(p-1)}{2}$ is of size $\frac{p(p+1)}{2}$ and vice-versa, it is enough to prove the statement for the case $|T|=\frac{p(p-1)}{2}$.

Since $PGL(2,p)$ acts triply transitively on the elements of the projective line we may assume that the three special directions are $(1,0),(0,1),(1,1)$. Moreover, it follows from the argument in Section \ref{sec3} that the set of intersections of $T$ with the horizontal lines is $\{0,1, \ldots, p-1 \}$ (since $|T|=\frac{p(p-1)}{2}$) and the same holds for the vertical lines. Moreover using a suitable affine transformations  along the axis we may assume that the order can be chosen to be $(0,1, \ldots, p-1 )$ along the vertical and $( p-1,p-2, \ldots, 1,0 )$ along the horizontal lines, respectively.

This shows that the first column does not contain any element of $T$ and since its first line contains $p-1$ elements we have $\{(0,i ) \mid 0<i\le p-1\} \subseteq T$. Using the same argument recursively one can prove that $T=S$.
\end{proof}

\section{Weighted sum of lines}\label{sec5}
The number of special directions of the set $S$ coincides with the number of non-Fourier roots of the characteristic function of $S$. This allows us to give a construction of sets with a given number of special directions. These sets are obtained as a linear combination of characteristic functions of lines (determining the special directions of the set) with rational coefficients.
It is easy to see that if $S$ is the weighted sum of lines of $k$ directions, then $S$ is equidistributed in every direction not determined by any line appearing in the sum.

One further aim is to present an alternative proof for the results of Section \ref{sec2} and Section \ref{sec4} using the following proposition.
We say that a function $f \colon \mathbb{F}_p^2 \to \C $ is equidistributed in a direction $d$ if the sum of the values of $f$ along the lines parallel to $d$ is constant.

\begin{prp}\label{prp frenkl}
Let $f \colon \mathbb{F}_p^2 \to \Q $ be a function. Assume $f$ is equidistributed in all but the following directions $d_1, \ldots , d_k$ ($k \ge 1$). Then $f$ can be written as the weighted sum of lines with rational weights: $$f=\sum_{j=1}^k\sum_{i=0}^{p-1} c_{j,i} 1_{l_{j,i}}, $$ where $c_{j,i} \in \Q$ and $l_{j,i}$ are lines determined by direction $d_j$ and for every $j\in \{1, \dots, k\}$ there is $i\in\{0, \dots, p-1\}$ $c_{j,i}\ne 0$. 
\end{prp}
\begin{proof}
We proceed by induction. Let $w_1$ be a function defined on the $\langle d_1 \rangle$-cosets such that $w_1$ takes the sum of the values of $f$ for the elements on each $\langle d_1 \rangle$-coset. Let $g_1(x)$ be a function of $\mathbb{F}_p^2$ defined as $g_1(x)=\frac{w_1(C)}{p}$, where $C$ is the $\langle d_1 \rangle$-coset containing $x$.
Since $f$ is not equidistributed in direction $d_1$, function $g_1(x)$ is not constant.

Clearly, $f_1:=f-g_1$ is a function equidistributed in all but $k-1$ directions.
If $k=1$, then $f_1$ is equdistributed in every direction and the sums along every line is zero. Then we claim that $f_1$ is zero. This can be seen from the fact that the Fourier transform of $f_1$ vanishes on every character.
Hence $f=g_1$, so it is of the form $\sum_{i=0}^{p-1} c_{1,i} 1_{l_{1,i}}$, where $l_{1,i}$ are determined by direction $d_j$ and $c_{1,i'}\ne 0$ for some $i'\in \{0,\dots, p-1\}$, since $g_1$ is not constant.

For $k\ne 1$ we get the statement for $f=f_1+g$ by using the inductive hypothesis for $f_1$.
\end{proof}
\begin{itemize}
\item
In particular we obtain the following explicit formula for the set discussed in the previous Section \ref{sec4}.

Let $\frac{c}{p}$ be the weight of lines defined by the equation $x=c$ and $\frac{-c}{p}$ the one of $y=c$. Further let $\frac{c}{p}$ be the weight of the vertical lines $y=x+c$. Now if $a<b$ we obtain that the sum of weight of these lines in this region is 1 and it is zero everywhere else.
\item
As a corollary of Proposition \ref{prp frenkl} one can see that there is no subset of $\mathbb{F}_p^2$, which is not equidistributed in exactly $2$ directions (see also Section \ref{sec2}). Similarly, one can also show the following Lam-Leung type result \cite{LL2000}, which is formulated for $\mathbb{F}_p \times \mathbb{F}_q$, where $p$ and $q$ are different primes. Let $S$ be a multiset, i.e, each value of the characteristic function of $S$ is a nonnegative integer, and suppose that there are at most two special directions of $S$. Then $S$ is a sum of weighted lines with nonnegative integer coefficients. The proofs of these results are analogous to the proof of Proposition 3.8 in \cite{KMSV}.
\end{itemize}




\section{Examples for four special directions}\label{sec6}
In this section we will try to find sets of smallest possible cardinality, having exactly $4$ special directions. 
For small prime $p\le 11$ we construct such sets of minimal cardinality according to Ghidelli's lower bound \cite{ghidelli}.

For a matrix $M$ let $M^{(k)}$ denote the matrix defined by $M^{(k)}(i,j)=M(i,j-k)$. Let $\underline{1}$ denote the all 1 row vector and let $e_i$ denote the vector which is 1 at its $i'th$ coordinate and zero everywhere else. Let  $L_j(v):=\sum _{i=0}^{p-2} \frac{p-i-1}{p}  v^{(ij)}$.
\begin{lem}\label{lem6.1}
\begin{enumerate}
    \item
Let $p$ be a prime and
let $v_j \in \mathbb{R}^{1 \times  \{0,1, \ldots, p-1 \}}$ be a row vector whose first coordinate is $1$ and $j+1$'th coordinate is $-1$.
Then \[ \frac{1}{p}\underline{1}+L_j(v_j):=\frac{1}{p}\underline{1} +\sum _{i=0}^{p-2} \frac{p-i-1}{p}  v_j^{(ij)}=e_1. \]
\item\label{refitem2}
 Let $k \in \mathbb{F}_p \setminus \{0\}$ with $k l \equiv j \pmod{p}$. Then \[ \frac{l}{p}\underline{1} +L_k(v_j):=\frac{l}{p}\underline{1} +\sum _{i=0}^{p-2} \frac{p-i-1}{p} v_j^{(ik)}=\sum_{a=0}^{l-1}e_1^{(ak)}. \]
\end{enumerate}
\end{lem}
\begin{proof}
\begin{enumerate}
    \item Easy calculation gives the result.
    \item It is easy to see that  $v_j=\sum_{a=0}^{l-1} v_k^{(ak)}$. The result follows from the fact that $L$ is additive and $L_k(v^{(m)})=L_k(v)^{(m)}$.
\end{enumerate}
\end{proof}
The way we are going to use this lemma is the following. We construct $\{\pm 1,0\}$ valued matrices whose row sums are all $0$. The previous lemma will be used simultaneously for the rows of these matrices.

We fix a $k \in \mathbb{F}_p \setminus \{0\}$ and we apply Lemma \ref{lem6.1} \eqref{refitem2}. Lemma \ref{lem6.1} treats
$\{\pm 1,0\}$ valued rows, which contain exactly one $1$'s and $-1$'s but since $L_k$  are linear operators we may apply it for some of these types of matrices. Now if we write a  $\{\pm 1,0\}$ valued row vector $v$, whose row sum is zero as the sum of $\{\pm 1,0\}$ valued row vectors ($v=\sum_{a=1}^c u_a$) with one $1$ and $-1$ entry, then $L_k(v)= L_k(\sum_{a=1}^c u_a)$.

Now for each $1 \le a \le c$ there is a $k_a$ such that $u_a=v_{i_a}^{(k_a)}$ for some $1 \le i_a-1,k_a \le p-1$. Finally, let $l_a k \equiv i_a \pmod{p}$. Then it follows from the previous conversation and Lemma \ref{lem6.1} that $L_k(u_a)$ will be a nonnegative integer valued vectors such that $\underline{1}^tL_k(u_a)=\sum_{a=1}^c l_a$.

The 4 directions used in the remainder of the section are $(1,0)$, $(0,1)$, $(1,1)$, $(1,-1)$. We build up sets, which are not equidistributed in these directions only.
It is clear that the sets presented in Figure \ref{Fig:triangles} can be constructed using weighted sums of lines in these directions.

\begin{figure}
\begin{center}
\begin{tikzpicture}
\matrix[matrix of nodes,nodes={draw=gray, anchor=center, minimum size=.5cm}, column sep=-\pgflinewidth, row sep=-\pgflinewidth] (A) at (-1,0) {
0 & 0 & 0 & 0 & 0 &0 & 0 & 0 & 0 & 0 &0 \\
0 & 0 & 0 & 0 & 0 &0 & 0 & 0 & 0 & 0 &1 \\
0 & 0 & 0 & 0 & 0 &0 & 0 & 0 & 0 & 1 &1 \\
0 & 0 & 0 & 0 & 0 &0 & 0 & 0 & 1 & 1 &1 \\
0 & 0 & 0 & 0 & 0 &0 & 0 & 1 & 1 & 1 &1 \\
0 & 0 & 0 & 0 & 0 &0 & 1 & 1 & 1 & 1 &1 \\
0 & 0 & 0 & 0 & 0 &1 & 1 & 1 & 1 & 1 &1 \\
0 & 0 & 0 & 0 & 1 &1 & 1 & 1 & 1 & 1 &1 \\
0 & 0 & 0 & 1 & 1 &1 & 1 & 1 & 1 & 1 &1 \\
0 & 0 & 1 & 1 & 1 &1 & 1 & 1 & 1 & 1 &1 \\
0 & 1 & 1 & 1 & 1 &1 & 1 & 1 & 1 & 1 &1 \\
};
\matrix[matrix of nodes,nodes={draw=gray, anchor=center, minimum size=.5cm}, column sep=-\pgflinewidth, row sep=-\pgflinewidth] (B) at (6,0)
{0 & 0 & 0 & 0 & 0 &0 & 0 & 0 & 0 & 0 &0 \\
0 & 1 & 0 & 0 & 0 &0 & 0 & 0 & 0 & 0 &0 \\
0 & 1 & 1 & 0 & 0 &0 & 0 & 0 & 0 & 0 &0 \\
0 & 1 & 1 & 1 & 0 &0 & 0 & 0 & 0 & 0 &0 \\
0 & 1 & 1 & 1 & 1 &0 & 0 & 0 & 0 & 0 &0 \\
0 & 1 & 1 & 1 & 1 &1 & 0 & 0 & 0 & 0 &0 \\
0 & 1 & 1 & 1 & 1 &1 & 1 & 0 & 0 & 0 &0 \\
0 & 1 & 1 & 1 & 1 &1 & 1 & 1 & 0 & 0 &0 \\
0 & 1 & 1 & 1 & 1 &1 & 1 & 1 & 1 & 0 &0 \\
0 & 1 & 1 & 1 & 1 &1 & 1 & 1 & 1 & 1 &0 \\
0 & 1 & 1 & 1 & 1 &1 & 1 & 1 & 1 & 1 &1 \\
};
\end{tikzpicture}
\end{center}
\caption{Two triangular sets with non-equidistributed directions (0,1), (1,0), (1,1) and (0,1), (1,0), (1,-1), respectively. }
\label{Fig:triangles}
\end{figure}


Now the difference of these two sets is also a weighted sum of suitable lines. This is presented in Figure \ref{fig:M11} and we denote it by $M_{11}$.

\begin{figure}
\begin{center}
\begin{tikzpicture}
\matrix[matrix of nodes,nodes={draw=gray, anchor=center, minimum size=.5cm}, column sep=-\pgflinewidth, row sep=-\pgflinewidth] (A) {
0 & 0 & 0 & 0 & 0 &0 & 0 & 0 & 0 & 0 &0 \\
0 & -1 & 0 & 0 & 0 &0 & 0 & 0 & 0 & 0 &1 \\
0 & -1 & -1 & 0 & 0 &0 & 0 & 0 & 0 & 1 &1 \\
0 & -1 & -1 & -1 & 0 &0 & 0 & 0 & 1 & 1 &1 \\
0 & -1 & -1 & 1 & -1 &0 & 0 & 1 & 1 & 1 &1 \\
0 & -1 & -1 & -1 & 1 &-1 & 1 & 1 & 1 & 1 &1 \\
0 & -1 & -1 & -1 & -1 &0 & 0 & 1 & 1 & 1 &1 \\
0 & -1 & -1 & -1 & 0 &0 & 0 & 0 & 1& 1 &1 \\
0 & -1 & -1 & 0 & 0 &0 & 0 & 0 & 0 & 1 &1 \\
0 & -1 & 0 & 0 & 0 &0 & 0 & 0 & 0 & 0 &1 \\
0 & 0 & 0 & 0 & 0 &0 & 0 & 0 & 0 & 0 &0 \\
};
\end{tikzpicture}
\end{center}
\caption{$M_{11}$.}
\label{fig:M11}
\end{figure}

Now let $N_{11}=M_{11}+M_{11}^{(5)}+C_{11}$, where $C_{11}$ is the matrix whose entries are all 0 except in the second and seventh columns which are constant $1$ and the fifth and last columns, which are constant -1.
An easy calculation shows that $N_{11}$ is the following matrix.

Now we apply the operator $L_{-2}$ simultaneously for the rows of $N_{11}$. Note that this can be realised as the sum of lines of the chosen directions.
One essential thing is that for those rows which contain more than one $1$'s (and $-1$'s) we have to find a pairing of these elements, which is indicated with colours.

We obtain the following $\{0,1\}$ matrix, which corresponds to the set we were looking for.

\begin{figure}
\begin{center}
\begin{tikzpicture}
\matrix[matrix of nodes,nodes={draw=gray, anchor=center, minimum size=.5cm}, column sep=-\pgflinewidth, row sep=-\pgflinewidth] (A) {
0 & \textcolor{red}{1} & 0 & 0 & \textcolor{blue}{-1} & 0 & \textcolor{blue}{1}  & 0 & 0 & 0 &\textcolor{red}{-1}
\\
0 & 0 & 0 & 0 & 0 &0 & 0 & 0 & 0 & 0 &0
\\
0 & 0 & \textcolor{red}{-1} & \textcolor{red}{1} & 0 &0 & 0 & \textcolor{blue}{-1} & 0 & \textcolor{blue}{1} &0
\\
0 & 0 & 0 & 0 & 0 &0 & 0 & -1 & 0& 1 &0
\\
0 & 1 & 0 & 0 & -1 &0 & 0 & 0 & 0 & 0 &0
\\
\textcolor{green}{1} & \textcolor{red}{1} & 0 & 0 & \textcolor{blue}{-1} & \textcolor{green}{-1} & \textcolor{blue}{1} & 0 & 0 & 0 &\textcolor{red}{-1}
\\
0 & 1 & 0 & 0 & -1 &0 & 0 & 0 & 0 & 0 &0
\\
0 & 0 & 0 & 0 & 0 &0 & 0 & -1 & 0& 1 &0
\\
0 & 0 & \textcolor{red}{-1} & \textcolor{red}{1} & 0 &0 & 0 & \textcolor{blue}{-1} & 0 & \textcolor{blue}{1} &0
\\
0 & 0 & 0 & 0 & 0 &0 & 0 & 0 & 0 & 0 &0
\\
0 & \textcolor{red}{1} & 0 & 0 & \textcolor{blue}{-1} & 0 & \textcolor{blue}{1}  & 0 & 0 & 0 &\textcolor{red}{-1}

\\
};

\matrix[matrix of nodes,nodes={draw=gray, anchor=center, minimum size=.5cm}, column sep=-\pgflinewidth, row sep=-\pgflinewidth] (B) at (7,0){
0 & 1 & 0 & 0 & 0 & 0 & 1 & 0 & 0 & 0 &0 \\
0 & 0 & 0 & 0 & 0 &0 & 0 & 0 & 0 & 0 &0 \\
0 & 1 & 0 & 1 & 1 &0 & 1 & 0 & 1 & 1 &1 \\
0 & 0 & 0 & 0 & 0 &0 & 0 & 0 & 0& 1 &0 \\
0 & 1 & 0 & 0 & 0 &0 & 1 & 0 & 1 & 0 &1 \\
1 & 1 & 0 & 0 & 0 & 0 & 1 & 1 & 0 & 1 &0 \\
0 & 1 & 0 & 0 & 0 &0 & 1 & 0 & 1 & 0 &1 \\
0 & 0 & 0 & 0 & 0 &0 & 0 & 0 & 0& 1 &0 \\
0 & 1 & 0 & 1 & 1 &0 & 1 & 0 & 1 & 1 &1 \\
0 & 0 & 0 & 0 & 0 &0 & 0 & 0 & 0 & 0 &0 \\
0 & 1 & 0 & 0 & 0 & 0 & 1 & 0 & 0 & 0 &0 \\
};
\end{tikzpicture}
\end{center}
\caption{The coloured pairs of 1's and -1's in $N_{11}$ on the left, and the corresponding set given by the process on the right.}
\end{figure}


A similar algorithm gives us the following sets for $p=5,7$ and $13$.
\begin{figure}
\begin{center}
\begin{tikzpicture}
\matrix[matrix of nodes,nodes={draw=gray, anchor=center, minimum size=.6cm}, column sep=-\pgflinewidth, row sep=-\pgflinewidth] (A) {
0 & 1 & 1 & 0 & 0 & 0 & 0 \\
0 & 1 & 1 & 0 & 0 & 1 & 0 \\
0 & 0 & 0 & 0 & 0 & 0 & 0 \\
0 & 0 & 0 & 0 & 0 & 0 & 0 \\
0 & 1 & 1 & 0 & 0 & 1 & 0 \\
0 & 1 & 1 & 0 & 0 & 0 & 0 \\
1 & 1 & 1 & 1 & 0 & 0 & 0 \\
};

\matrix[matrix of nodes,nodes={draw=gray, anchor=center, minimum size=.6cm}, column sep=-\pgflinewidth, row sep=-\pgflinewidth] (B) at (5,0) {
0 & 0 & 0 & 0 & 0 \\
0 & 0 & 1 & 0 & 0  \\
0 & 1 & 1 & 1 & 0  \\
0 & 0 & 1 & 0 & 0  \\
0 & 0 & 0 & 0 & 0 \\
};
\end{tikzpicture}
\end{center}
\caption{Examples for sets of smallest cardinality that have 4 special directions for $p=7$ (on the left) and $p=5$ (on the right).}
\end{figure}

Note that if we fix the prime $p$ and the number of special directions $n$, then Ghidelli's result gives us a lower bound for the subsets of $\mathbb{F}_p^2$ having exactly $k$ special directions. The previous examples for $p=5,7,11$ meet this lower bound. However this is not the case for the next example $p=13$, which seems to be optimal using this method.

The original method only gives us a multiset, where the sum of the weights is 65, which can easily be modified by subtracting $\underline{1}$ from those lines which contain weight 2 as well and adding  $\underline{1}^t$ to those columns which are currently empty. This does not modify the sum of the values but makes the following matrix to a $\{0,1 \}$-matrix.

\begin{figure}
\begin{center}
\begin{tikzpicture}
\matrix[matrix of nodes,nodes={draw=gray, anchor=center, minimum size=.5cm}, column sep=-\pgflinewidth, row sep=-\pgflinewidth] (A) {
1 & 1 & 0 & 0 & 1 &0 & 0 & 1 & 1 & 0 & 0 & 0 & 0 \\

1 & 1 & 0 & 1 & 1 &1 & 0 & 1 & 1 & 1 &0 &0 &1\\

0 & 0 & 0 & 0 & 1 &0 & 0 & 0 & 0 & 0 & 0 & 0 & 0 \\

0 & 1 & 0 & 1 & 1 &1 & 0 & 1 & 0 & 0 & 1 & 1 & 0 \\

0 & 0 & 0 & 0 & 1 &0 & 0 & 0 & 0 & 0 & 0 & 0 & 0 \\

1 & 1 & 0 & 1 & 1 &1 & 0 & 1 & 1 & 1 &0 &0 &1\\

1 & 1 & 0 & 0 & 1 &0 & 0 & 1 & 1 & 0 & 0 & 0 & 0 \\

0 & 1 & 0 & 0 & 0 &0 & 0 & 1 & 0 & 0 & 0 & 0 & 0 \\

0 & 0 & 0 & 0 & 0 &0 & 0 & 0 & 0 & 0 & 0 & 0 & 0 \\

1 & 1 & 0 & 1 & 1 &2 & 0 & 1 & 1 & 1 & 1 & 1 & 1 \\

1 & 1 & 0 & 1 & 1 &2 & 0 & 1 & 1 & 1 & 1 & 1 & 1 \\

0 & 1 & 0 & 0 & 0 &0 & 0 & 1 & 0 & 0 & 0 & 0 & 0 \\
};
\end{tikzpicture}
\end{center}
\caption{Example of 65-element set with 4 special directions in $\mathbb{F}_{13}^2$.}
\end{figure}
The question remains whether there exists $4*13=52$-element subset of $\mathbb{F}_{13}^2$ determining exactly $4$ special directions.
\section{Open problems}\label{sec7}
\begin{enumerate}
    \item
Is there a set in $\mathbb{F}_p^2$ that is equidistributed in exactly $d$ directions for every $d \le p-2$? We have seen that this is not the case for $d=p-1$ since sets which are equidistributed in $p-1$ directions are unions of parallel lines so these are equidistributed in $p$ directions at least.

It follows from the result of R\'edei \cite{redei} that there is a gap in the possible number of special directions for subsets of $\mathbb{F}_p^2$ of cardinality $p$. Further, the result of G\'acs \cite{gacs} shows that this is not a unique gap since sets (of cardinality $p$ having more than $\frac{p+3}{2}$ special directions determine at least $\lfloor 2 \frac{p-1}{3}+1\rfloor$ special directions.

However, it is not hard to see that for $p=3,5,7$ there is no such a gap if the cardinality of the set is divisible by $p$ . 

 \item What is the minimal size of a set in $\mathbb{F}_p^2$, which is equidistributed in at most $k$ directions?
 In particular, is it true that Ghidelli's bound is tight \cite{ghidelli}, i.e., is it possible to construct sets of cardinality $kp$ which have $\lceil \frac{p+k+2}{k+1}\rceil$ special directions?

 Even for $p=13$ this question is still open.
\end{enumerate}
\section*{Acknowledgement}
G. Kiss was supported by the J\'anos Bolyai Research Grant, the New National Excellence Program ÚNKP-22-5-ELTE-1154 New National Excellence Program of the Ministry for Culture and
Innovation and the Hungarian National Research, Development and Innovation Office - NKFIH (grant no. K124749, no. K142993).

G. Somlai is a Fulbright research fellow at the Graduate Center of the City University of New York. This research exchange program is also supported by the Magyar Állami Eötvös Ösztöndíj.

G. Somlai is a J\'anos Bolyai fellowship holder and supported by the OTKA grant no. SNN 132625.

\end{document}